\documentclass[12pt]{article}

\usepackage{amsmath,amsthm,amsfonts,amssymb,color,graphicx,tikz,comment, pgfplots, xfrac}
\usepackage[colorlinks,citecolor=blue,urlcolor=blue]{hyperref}
\usetikzlibrary{shapes.geometric,positioning}
\excludecomment{codes}

\theoremstyle{definition}
\newtheorem{theorem}{Theorem}[section]

\newtheorem{lemma}[theorem]{Lemma}

\newtheorem{remark}[theorem]{Remark}
\numberwithin{equation}{section}

\def\cD{\mathcal{D}}

\def\cF{\mathcal{F}}

\def\cH{\mathcal{H}}

\def\bE{\mathbb{E}}
\def\bN{\mathbb{N}}
\def\bR{\mathbb{R}}
\def\R{\mathbb{R}}

\newcommand{\ud}{\ensuremath{ \mathrm{d} }}

\newcommand{\HH}{
  \begin{tikzpicture}[scale=0.95, transform shape,baseline={([yshift=-.5ex]current bounding box.center)}]
    \begin{axis}[
      axis lines = left,
      xtick={0,       0.25,           0.5,            0.75,           1},
      xticklabels={0, $\sfrac{1}{4}$, $\sfrac{1}{2}$, $\sfrac{3}{4}$, $1$},
      ytick={0,       0.25,           0.5,            0.75,           1},
      yticklabels={0, $\sfrac{1}{4}$, $\sfrac{1}{2}$, $\sfrac{3}{4}$, $1$},
      every axis x label/.style={at={(current axis.right of origin)},anchor=west,below=8pt,xshift=0.8em},
      every axis y label/.style={at={(current axis.north west)},left=0pt, yshift=0.4em},
      xlabel=$H_0$,
      ylabel=$H$,
      xmin=0.45,  xmax=1.10,
      ymin=-0.05, ymax=0.60,
      x=9em, y=9em
      ]
      \filldraw[gray]
        (axis cs:0.75, 0) --
        (axis cs:1,    0) --
        (axis cs:1,    0.5) --
        (axis cs:0.5,  0.5) --
        (axis cs:0.5,  0.25) --
        (axis cs:0.75, 0);

      \draw[dashed, very thick]
        (axis cs: 0.0,0.0) -- (axis cs: 1.0,0.0)
        (axis cs: 0.5,0.0) -- (axis cs: 0.5,1.0)
        (axis cs: 0.0,0.5) -- (axis cs: 1.0,0.5)
        (axis cs: 1.0,0.0) -- (axis cs: 1.0,1.0)
        (axis cs: 0.0,0.0) -- (axis cs: 0.0,1.0)
        (axis cs: 0.0,0.75) -- (axis cs: 0.75,0.0)
        (axis cs: 0.0,1.0) -- (axis cs: 1.0,1.0);
    \end{axis}
  \end{tikzpicture}}
\begin{document}

\title{Parabolic Anderson model with rough noise in space and rough initial conditions}
\author{
  Raluca M. Balan\footnote{
    University of Ottawa, Department of Mathematics and Statistics,
    150 Louis Pasteur Private, Ottawa, Ontario, K1G 0P8, Canada.
    E-mail address: \url{rbalan@uottawa.ca}.}
  \footnote{Research supported by a grant from the Natural Sciences and Engineering Research Council of Canada.}
  \and
  Le Chen\footnote{Corresponding author. Auburn University,
    Department of Mathematics and Statistics, 203 Parker Hall, Auburn, Alabama 36849, USA.
    E-mail address: \url{le.chen@auburn.edu}.}
  \and
  Yiping Ma\footnote{The Ottawa Hospital,
    2475 Don Reid Drive Ottawa, Ontario, K1H1E2, Canada.
    E-mail address: \url{yiping.ivyma@gmail.com}.}
  }
\date{June 20, 2022}
\maketitle

\begin{abstract}
  \noindent In this note, we consider the parabolic Anderson model on $\bR_{+} \times \bR$, driven
  by a Gaussian noise which is fractional in time with index $H_0>1/2$ and fractional in space with
  index $0<H<1/2$ such that $H_0+H>3/4$. Under a general condition on the initial data, we prove the
  existence and uniqueness of the mild solution and obtain its exponential upper bounds in time for
  all $p$-th moments with $p\ge 2$.
\end{abstract}


\vspace{1mm}


\section{Introduction}
In this paper, we study the {\it parabolic Anderson model} (PAM):
\begin{align} \label{E:PAM}
  \begin{cases}
    \dfrac{\partial u}{\partial t} (t,x)=\dfrac{1}{2} \dfrac{\partial u}{\partial x^2} u(t,x)+u(t,x)\dot{W}(t,x) \quad \quad t>0,\: x \in \bR,\\
    u(0,\cdot) = \mu_0,
  \end{cases}
\end{align}
with initial condition given by a non-negative Borel measure $\mu_0$ on $\bR$ such that
\begin{equation} \label{E:cond-mu0}
  \int_{\bR}e^{-a x^2}\mu_0(\ud x)<\infty \quad \text{for all $a>0$.}
\end{equation}
Initial conditions of this type (called {\it rough initial conditions}) were introduced in
\cite{chen.dalang:15:moments} and were considered later for the {\it stochastic heat equation} in
various settings; see, e.g., \cite{balan.chen:18:parabolic, chen.huang:19:comparison,
chen.kim:19:nonlinear} and references therein. The noise $\dot W$ is assumed to be a centered
Gaussian noise that is fractional in time and space with indices $H_0$, respectively $H$ in the
following range:
\begin{gather} \label{E:H-cond}
  (H_0,H)\in \left(\sfrac{1}{2},1\right) \times \left(0,\sfrac{1}{2}\right) \quad \text{and} \quad
  H+H_0>\sfrac{3}{4}, \quad \text{i.e.,} \text{\vspace{2em} \HH }.
\end{gather}

Rigorously, $\{W(\varphi);\varphi \in \cD(\bR_{+} \times \bR)\}$ is a zero-mean Gaussian process
with covariance\footnote{In this note, we denote by $\cF \varphi=\int_{\bR^d}e^{-i \xi \cdot x}
\varphi(x)\ud x$ the Fourier transform of a function $\varphi \in L^1(\bR^d)$.}:
 \begin{align*}
  \bE[W(\varphi)W(\psi)]=\alpha_{H_0}\int_{\bR_{+}^2} \int_{\bR}|t-s|^{2H_0-2} \cF \varphi(t,\cdot)(\xi)\overline{\cF \psi(s,\cdot)(\xi)}\mu(\ud \xi) \ud t\ud s=: \langle \varphi,\psi \rangle_{\cH},
\end{align*}
where $\alpha_{H_0}=H_0(2H_0-1)$ and $\mu(\ud \xi)=c_{H}|\xi|^{1-2H}\ud \xi $, with
$c_H=\Gamma(2H+1)\sin(\pi H)/(2\pi)$. Since $H<1/2$, we say that $W$ is {\it rough} in space.

We denote by $\cH$ the completion of $\cD(\bR_{+} \times \bR)$ with respect to the inner product
$\langle \cdot,\cdot \rangle_{\cH}$. Then $W=\{W(\varphi)\}_{\varphi \in \cH}$ is an isonormal
Gaussian process and we can use Malliavin calculus to define and analyze the solution to
\eqref{E:PAM}. We say that a process $u=\{u(t,x);t> 0,x\in \bR\}$ is a {\em Skorohod solution} of
\eqref{E:PAM} if it is adapted with respect to the filtration induced by $W$, and for all $t>0$ and
$x \in \bR$,
\begin{equation} \label{E:def-sol}
  u(t,x)=J_0(t,x)+\int_0^t \int_{\bR}G(t-s,x-y)u(s,y)W(\delta s,\delta y),
\end{equation}
where $ J_0$ is the solution to the homogeneous heat equation, i.e.,
\begin{align} \label{E:J0}
  J_0(t,x) := \int_{\bR}G(t,x-y)\mu_0(\ud y) \quad \text{with} \quad
  G(t,x)    = (2\pi t)^{-1/2}e^{- \frac{x^2}{2t}}.
\end{align}

The stochastic integral in \eqref{E:def-sol} is interpreted in the Skorohod sense, i.e. it is given
by the divergence operator from Malliavin calculus. We refer the reader to Section 1.3 of
\cite{nualart:06:malliavin} for the definition of this operator, and to
\cite{balan.chen:18:parabolic,hu.le:19:joint,hu.huang.ea:18:parabolic} for similar developments.
\bigskip

The following theorem is the main result of the present article.

\begin{theorem} \label{T:main}
  If $(H_0,H)$ satisfy \eqref{E:H-cond} and if $\mu_0$ satisfies \eqref{E:cond-mu0}, then equation
  \eqref{E:PAM} has a unique solution $u$ and this solution satisfies: for all $p \geq 2$, $t>0$ and
  $x \in \bR$,
  \begin{align} \label{E:main}
    \bE\left(|u(t,x)|^p\right) \leq C_1^p J_0^p(t,x) \exp\left(C_2 p^{\frac{H+1}{H}}t^{\frac{2H_0+H-1}{H}} \right),
  \end{align}
  where $C_1>0$ and $C_2>0$ are some constants which depend on $H_0$ and $H$.
\end{theorem}

The novelty of our result is the fact that we consider rough initial condition. One prominent
example is the case $\mu_0=\delta_0$ where $\delta_0$ is the Dirac delta measure; see, e.g.,
\cite{amir.corwin.ea:11:probability}.

\begin{remark} \label{R:Bridge}
When the initial condition is a bounded function or a constant, the well-posedness was obtained by
X. Chen; see Theorem 1.2 of \cite{chen:19:parabolic}. The extension of X. Chen's work to rough
initial conditions is highly nontrivial. As we will see in the proof of Theorem \ref{T:main} below,
to estimate the moments of the solution, we need to compute a spatial integral and a time integral.
In \cite{chen:19:parabolic}, X. Chen uses the Laplace transform to handle the time integral first
and then computes the spatial integral. This method does not work for the rough initial data. In a
nutshell, moving from bounded initial data to rough initial data, one essentially changes the
underlying motion from the Brownian motion to the Brownian bridge. The time increments no longer
take the linear difference form; see Lemma \ref{lem32-BC}. To overcome this difficulty, we compute
the spatial integral first and then use Lemma \ref{L:Int} to compute the time integral. This leads
to a complicated expression which contains a product $\gamma_n$ of ratios of gamma functions; see
\eqref{E:int}. The most delicate part is to estimate $\gamma_n$. For this, we develop some novel
combinatorial techniques.
\end{remark}

Recently, Hu and L\^e obtained in Theorem 3.2 of \cite{hu.le:19:joint} both the well-posedness and
the following moment asymptotics: \footnote{This is relation (3.6) of \cite{hu.le:19:joint} with
$\alpha_0=2-2H_0$ and $\alpha=2H-1$; see also Remark 3.5 (ii) {\it ibid.} }
\begin{equation} \label{E:hu-le-bound}
       \bE\left(|u(t,x)|^p\right)
  \leq C_1^p t^{-p(\beta+\frac{2H-1}{4})}\exp \big(C_2 p^{\frac{H+1}{H}}t^{\frac{2H_0+H-1}{H}} \big),
  \quad \text{for all $p\ge 2$, $t>0$, $x\in\R$,}
\end{equation}
under weaker conditions on $\mu_0$, namely, $\mu_0$ is a Borel measure such that
\begin{equation} \label{E:HL-cond}
  \int_{\bR}\left(1+|\xi|^{-(H-1/2)}\right)e^{-t|\xi|^2} |\cF \mu_0(\xi)|\ud \xi \leq C t^{-\beta}
  \quad \text{for all $t>0$},
\end{equation}
for some $C>0$ and $\beta<H_0$, where $\cF \mu_0$ is the Fourier transform of $\mu_0$. Condition
\eqref{E:HL-cond} is more restrictive than \eqref{E:cond-mu0}; see Remark \ref{R:Init} for one
example. While our exponent in \eqref{E:main} recovers that of Hu and L\^e in \eqref{E:hu-le-bound},
the factor $J_0^p(t,x)$ in \eqref{E:main} looks more natural than the corresponding factor in
\eqref{E:hu-le-bound}. For example, when the initial condition is a bounded function (resp. the
delta initial measure), then as $t\to 0$, the factor $J_0^p(t,x)$ will not blow up (resp. blows up
at the exact rate $t^{-d/2}$ for $x=0$ and will not blow up for $x\ne 0$). But it is not clear
whether the factor $t^{-p(\beta+(2H-1)/4)}$ in \eqref{E:hu-le-bound} would blow up or not, which
depends on the sign of the exponent of $t$.

\begin{remark} \label{R:Init}
  Our condition on the initial data allows growing tails, for example, $\mu(\ud x) = x^2 \ud x$. In
  this case, $\int_\R e^{-ax^2} x^2\ud x= \sqrt{\pi}\: 2^{-1}a^{-3/2}<\infty$ for all $a>0$. Hence,
  condition \eqref{E:cond-mu0} is satisfied. But this initial condition cannot satisfy condition
  \eqref{E:HL-cond} because $\mathcal{F}[ x^2 ](\xi) = \delta''(\xi)$ (in the generalized sense,
  see, e.g., Theorem 7.4 of \cite{rudin:91:functional}), which is a genuine distribution and hence
  does not have module or absolute value.
\end{remark}

Finally, this paper can also be viewed as a continuation of \cite{balan.chen:18:parabolic} where the
case of rough initial conditions and $(H_0,H)\in (1/2,1)^2$ was covered.

\section{Proof of Theorem \ref{T:main}}

We denote by $I_n :\cH^{\otimes n} \to \cH$ the multiple Wiener integral of order $n$ with respect
to $W$. It is known that the solution $u$ exists if and only if $\sum_{n\geq 1}I_n(f_n(\cdot,t,x))$
converges in $L^2(\Omega)$, and in this case the solution has the Wiener chaos expansion:
\begin{gather*}
  u(t,x)=J_0(t,x)+\sum_{n\geq 1}I_n(f_n(\cdot,t,x)) \quad \text{with} \\
  f_n(t_1,x_1,\ldots,t_n,x_n,t,x)=\prod_{j=1}^{n}G(t_{j+1}-t_{j},x_{j+1}-x_{j})
  J_0(t_1,x_1)1_{\{0<t_1<\ldots<t_n<t\}},
\end{gather*}
and $t_{n+1}=t$ and $x_{n+1}=x$; see for instance
\cite{balan.chen:18:parabolic,hu.huang.ea:18:parabolic}. By the orthogonality of the terms in this
series, the necessary and sufficient condition for the existence of solution is:
\begin{equation} \label{ns-cond}
  \sum_{n\geq 1}n! \|\widetilde{f}_n(\cdot,t,x)\|_{\cH^{\otimes n}}^2 <\infty,
\end{equation}
where $\widetilde{f}_n(\cdot,t,x)$ is the symmetrization of $f_n(\cdot,t,x)$, defined by:
\begin{align*}
  \widetilde{f}_n(t_1,x_1,\ldots,t_n,x_n,t,x)=\frac{1}{n!}\sum_{\rho \in S_n}
  f_n(t_{\rho(1)},x_{\rho(1)},\ldots,t_{\rho(n)},x_{\rho(n)},t,x),
\end{align*}
where $S_n$ is the set of permutations of $\{1,\ldots,n\}$. For any ${\bf t}=(t_1,\ldots,t_n) \in
[0,t]^n$, ${\bf s}=(s_1,\ldots,s_n) \in [0,t]^n$, we denote
\begin{align*}
 \psi_{t,x}^{(n)}({\bf t},{\bf s})=  (n!)^2\int_{\bR^d}\mu(\ud \xi_1) \ldots \mu(\ud \xi_n)\:
        & \cF \widetilde{f}_n(t_1,\cdot,\ldots,t_n,\cdot,t,x)(\xi_1,\ldots,\xi_n) \\
 \times & \overline{\cF \widetilde{f}_n(s_1,\cdot,\ldots,s_n,\cdot,t,x)(\xi_1,\ldots,\xi_n)}.
\end{align*}

 We will use the following result.

\begin{lemma}[Lemma 3.2 of \cite{balan.chen:18:parabolic}]
\label{lem32-BC}
If $0<t_{\rho(1)}<\ldots<t_{\rho(n)}<t=:t_{\rho(n+1)}$, then
\begin{align*}
      \psi_{t,x}^{(n)}({\bf t},{\bf s})
  \le J_0^2(t,x)\int_{\bR^n} \prod_{k=1}^{n}\exp \left\{-\frac{t_{\rho(k+1)}-t_{\rho(k)}}{t_{\rho(k+1)}t_{\rho(k)}}
      \left|\sum_{j=1}^{k}t_{\rho(j)}\xi_j\right|^2\right\}\mu(\ud \xi_1)\ldots \mu(\ud \xi_n).
\end{align*}
\end{lemma}

We will also use the following estimate, which is a consequence of H\"older's inequality with
$p=1/H$, and the {\it Littlewood-Hardy-Sobolev inequality} (see, e.g., \cite{memin.mishura.ea:01:inequalities}):
\begin{equation} \label{E:LHS}
       \alpha_{H_0}^n\int_{\bR_{+}^{2n}} \prod_{j=1}^{n}|t_j-s_j|^{2H_0-2}\varphi({\bf t}) \varphi({\bf s})\ud{\bf t} \ud{\bf s}
  \leq b_{H_0}^n \left(\int_{\bR_{+}^n} |\varphi({\bf t})|^{1/H_0} \ud{\bf t}\right)^{2H_0}.
\end{equation}

\begin{lemma} \label{L:A}
  For $n\ge 2$ and $x_1,\ldots,x_n \in \bR_+$, it holds that
  \begin{equation} \label{E:id}
    S_n:=x_1\prod_{k=2}^n(x_k+x_{k-1})=\sum_{a \in A_n}\prod_{j=1}^n x_j^{a_j},
  \end{equation}
  where $A_n$ is a set of indices $a=(a_1,\ldots,a_n)$ such that ${\rm card}(A_n)=2^{n-1}$ and
  \begin{subequations} \label{E:a}
  \begin{gather}
    a_1\in\{1,2\}, \quad a_{n} \in \{0,1\}, \quad a_2,\ldots,a_{n-1} \in \{0,1,2\}, \quad \\
    \sum_{j=1}^{i} a_j\in \left\{i,i+1\right\} \quad \text{for $i=1,\cdots, n-1$,} \sum_{j=1}^{n}
    a_j=n, \\
    a_i+a_{i+1} \in \{1,2,3\} \quad \text{for $i=2,\ldots, n-2$}, \\
    a_1+a_2\in \{2,3\} \quad \text{and} \quad a_{n-1}+a_n \in \{1,2\}.
  \end{gather}
  \end{subequations}
\end{lemma}
\begin{proof}
  Clearly it holds for $n=2$. Assume that the statement is true for $n\ge 2$. Then
  \begin{align*}
       S_{n+1}
    =  S_n(x_n+x_{n+1})
    =  \sum_{a\in A_n} \prod_{i=1}^{n}x_i^{a_i}x_n^{a_n+1}+\sum_{a\in A_n}\prod_{i=1}^n x_i^{a_i} x_{n+1}
    =: \sum_{a'\in A_{n+1}} \prod_{i=1}^{n+1}x_i^{a_i'}.
  \end{align*}
  The statement for $n+1$ follows by considering separately two cases: (i) $a_1'=a_1,\ldots,
  a_{n-1}'=a_{n-1}, a_n'=a_n+1, a_{n+1}'=0$; and (ii) $a_1'=a_1,\ldots,,a_n'=a_n,a_{n+1}'=1$.
\end{proof}
\begin{remark} \label{R:paths}
  Lemma \ref{L:A} can also be proved using a path representation. Since this representation will be used in the proof of Theorem \ref{T:main}, we explain it here. To each monomial $x_1^{a_1}\ldots x_n^{a_n}$ in the
  expansion of $S_n$ one can associate a path starting from $(1,1)$ and going to $(n,n)$ or $(n,n-1)$, depending on whether $x_n$ is present or
  absent in the monomial. This path is composed of $n-1$ segments, which correspond to exponents
  $a_1,\ldots,a_{n-1}$ (in this order) and are constructed as follows:        \\
  -if the exponent is $0$, the path moves 1 unit to the right and 2 units up; \\
  -if the exponent is $1$, the path moves 1 unit to the right and 1 unit up; \\
  -if the exponent is $2$, the path moves 1 unit to the right and 0 units up. \\
  See Figure \ref{F:paths} for an illustration of this correspondence with $n=4$. See also Figure
  \ref{F:a} for the properties in \eqref{E:a}.
\end{remark}
\begin{figure}[htpb]

  \newcommand{\SamplePath}[3]{
    \begin{tikzpicture}[scale=1, transform shape]
      \tikzset{>=latex}
      \begin{axis}[
        axis lines = left,
        x=2em,y=2em,
        xtick={1,2,3,4},
        xtick={1,2,3,4},
        yticklabel style = {font=\tiny,xshift=0.5ex},
        xticklabel style = {font=\tiny,yshift=0.5ex},
        legend style={at={(axis description cs: 0.1,1.1)},anchor=north},
        legend style={draw=none},
        xmin=0.7, xmax=4.3,
        ymin=0.7, ymax=4.3,
        title={#1},  title style={at={(0.3,0.7)}},
        ]
        \filldraw (axis cs: 1,1) circle[radius=2pt] node [above,yshift=+2pt] {$x_1$};
        \filldraw (axis cs: 2,2) circle[radius=2pt] node [above,yshift=+2pt] {$x_2$};
        \filldraw (axis cs: 3,3) circle[radius=2pt] node [above,yshift=+2pt] {$x_3$};
        \filldraw (axis cs: 4,4) circle[radius=2pt] node [left, xshift=-2pt] {$x_4$};
        \filldraw (axis cs: 2,1) circle[radius=2pt] node [right,xshift=2pt]  {$x_1$};
        \filldraw (axis cs: 3,2) circle[radius=2pt] node [right,xshift=2pt]  {$x_2$};
        \filldraw (axis cs: 4,3) circle[radius=2pt] node [below,yshift=-3pt] {$x_3$};
        \draw[dashed,red,thick] (axis cs: 1,1)
                             -- (axis cs: 4,4)
                             -- (axis cs: 4,3)
                             -- (axis cs: 2,1)
                             -- (axis cs: 1,1);
       \node at (axis cs: 3.6,1.2) {\small\bf\em \textcolor{black}{#2}};
       #3;
    \end{axis}
    \end{tikzpicture}}

  \begin{center}

    \SamplePath{$x_1^2x_2x_3$}{2110}{\draw[very thick,->] (axis cs: 1,1) -- (axis cs: 2,1) -- (axis cs: 3,2)-- (axis cs: 4,3)}
    \SamplePath{$x_1^2x_2x_4$}{2101}{\draw[very thick,->] (axis cs: 1,1) -- (axis cs: 2,1) -- (axis cs: 3,2)-- (axis cs: 4,4)}
    \SamplePath{$x_1^2x_3^2$ }{2020}{\draw[very thick,->] (axis cs: 1,1) -- (axis cs: 2,1) -- (axis cs: 3,3)-- (axis cs: 4,3)}
    \SamplePath{$x_1^2x_3x_4$}{2011}{\draw[very thick,->] (axis cs: 1,1) -- (axis cs: 2,1) -- (axis cs: 3,3)-- (axis cs: 4,4)}

    \SamplePath{$x_1x_2^2x_3$ }{1210}{\draw[very thick,->] (axis cs: 1,1) -- (axis cs: 2,2) -- (axis cs: 3,2)-- (axis cs: 4,3)}
    \SamplePath{$x_1x_2^2x_4$ }{1201}{\draw[very thick,->] (axis cs: 1,1) -- (axis cs: 2,2) -- (axis cs: 3,2)-- (axis cs: 4,4)}
    \SamplePath{$x_1x_2x_3^2$ }{1120}{\draw[very thick,->] (axis cs: 1,1) -- (axis cs: 2,2) -- (axis cs: 3,3)-- (axis cs: 4,3)}
    \SamplePath{$x_1x_2x_3x_4$}{1111}{\draw[very thick,->] (axis cs: 1,1) -- (axis cs: 2,2) -- (axis cs: 3,3)-- (axis cs: 4,4)}

  \end{center}

  \caption{The eight paths for $n=4$, from $(1,1)$ to either $(4,3)$ or $(4,4)$, correspond to the
  eight monomials in the expansion of $S_4=x_1(x_1+x_2)(x_2+x_3)(x_3+x_4)$. All paths should stay in
the dashed envelope. The four digits correspond to the value of $(a_1,\cdots,a_4)$. }

  \label{F:paths}
\end{figure}
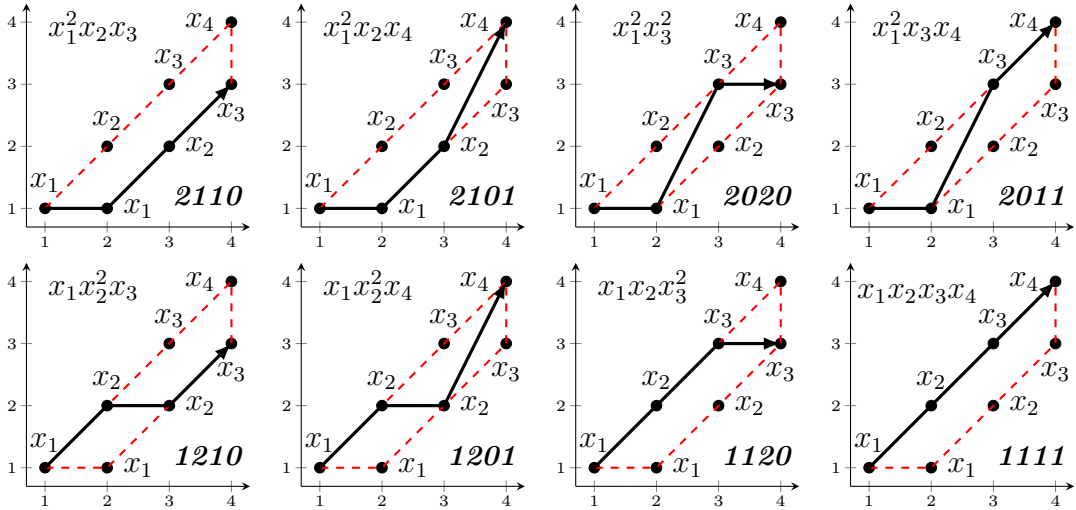
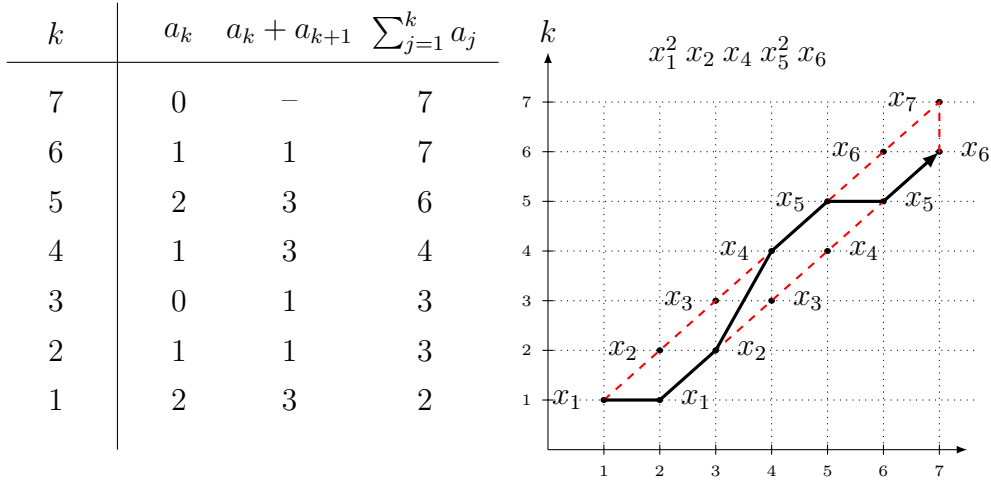
\begin{figure}[htbp!]
  \centering
  \begin{center}
    \begin{tikzpicture}[scale=1, transform shape, x=1.8em, y=1.6em]
      \tikzset{>=latex}

      \draw[->] (0,0) --++(7.5,0); 
      \foreach \x in {1,...,7}{
          \draw (\x,0.1)--++(0,-0.2) node [below] {\tiny $\x$};
          \draw[dotted] (\x,0)--++(0,7);
      }
      \draw[->] (0,0) --++(0,8) node[above] {$k$};
      \foreach \x in {1,...,7}{
          \draw (0.1,\x)--++(-0.2,0) node [left] {\tiny $\x$};
          \draw[dotted] (0,\x) --++ (7.7,0);
      }

      \node at (3.4,8) {$x_1^2\: x_2\: x_4\: x_5^2\: x_6$};

      \foreach \x in {1,...,7}{
        \filldraw (\x,\x) circle (-0.05) node[left,xshift=-4pt] {$x_{\x}$};
      }
      \foreach \x in {1,...,6}{
        \pgfmathsetmacro\result{\x}
        \filldraw (\x+1,\x) circle (-0.05) node [right, xshift=4pt] {$x_{\result}$};
      }

      \draw[dashed,red, thick] (1,1) -- (7,7) -- (7,6) -- (2,1) -- (1,1);
      \draw[->,very thick] (1,1) -- (2,1) -- (3,2) -- (4,4) -- (5,5) -- (6,5) -- (7,6);


      \def\displace{-2.2}
      \draw (4.4*\displace,7.8) --++ (-4*\displace,0);
      \draw (3.5*\displace,9) -- (3.5*\displace,0) ;

      \node[] () at (4*\displace,1) {$1$};
      \node[] () at (4*\displace,2) {$2$};
      \node[] () at (4*\displace,3) {$3$};
      \node[] () at (4*\displace,4) {$4$};
      \node[] () at (4*\displace,5) {$5$};
      \node[] () at (4*\displace,6) {$6$};
      \node[] () at (4*\displace,7) {$7$};
      \node[] () at (4*\displace,8.4) {$k$};

      \node[] () at (3*\displace,1) {$2$};
      \node[] () at (3*\displace,2) {$1$};
      \node[] () at (3*\displace,3) {$0$};
      \node[] () at (3*\displace,4) {$1$};
      \node[] () at (3*\displace,5) {$2$};
      \node[] () at (3*\displace,6) {$1$};
      \node[] () at (3*\displace,7) {$0$};
      \node[] () at (3*\displace,8.4) {$a_k$};

      \node[] () at (2.1*\displace,1) {$3$};
      \node[] () at (2.1*\displace,2) {$1$};
      \node[] () at (2.1*\displace,3) {$1$};
      \node[] () at (2.1*\displace,4) {$3$};
      \node[] () at (2.1*\displace,5) {$3$};
      \node[] () at (2.1*\displace,6) {$1$};
      \node[] () at (2.1*\displace,7) {--};
      \node[] () at (2.1*\displace,8.4) {$a_k+a_{k+1}$};

      \node[] () at (1.0*\displace,1) {$2$};
      \node[] () at (1.0*\displace,2) {$3$};
      \node[] () at (1.0*\displace,3) {$3$};
      \node[] () at (1.0*\displace,4) {$4$};
      \node[] () at (1.0*\displace,5) {$6$};
      \node[] () at (1.0*\displace,6) {$7$};
      \node[] () at (1.0*\displace,7) {$7$};
      \node[] () at (1.0*\displace,8.4) {$\sum_{j=1}^{k} a_j$};
    \end{tikzpicture}
  \end{center}

  \caption{Illustrations of properties in \eqref{E:a} with $n=7$. Each monomial in the expansion of
  \eqref{E:id} corresponds to a path from $(1,1)$ to either $(7,7)$ or $(7,6)$. The dots indicate the
possible choices for the position of the path. }

  \label{F:a}
\end{figure}

Now we are ready to prove Theorem \ref{T:main}.

\begin{proof}[Proof of Theorem \ref{T:main}]
Denote $J_n(t,x)=I_n(f_n(\cdot,t,x))$. Let $C$ be a constant which depends on $H_0$ and $H$ and may
be different from line to line. The proof consists of five steps: \medskip

{\noindent\em Step 1.} Note that $\psi_{t,x}^{(n)}({\bf t},{\bf s}) \leq \psi_{t,x}^{(n)}({\bf
t},{\bf t})^{1/2} \psi_{t,x}^{(n)}({\bf s},{\bf s})^{1/2}$ by the Cauchy-Schwarz inequality.
Combining this with inequality \eqref{E:LHS}, we obtain:
\begin{align} \label{E:e1}
  \nonumber \bE\left(|J_n(t,x)|^2\right)
  & =    n! \|\widetilde{f}_n(\cdot,t,x)\|_{\cH^{\otimes n}}^2
    =    \frac{1}{n!}\alpha_{H_0}^n \int_{[0,t]^{2n}}\prod_{j=1}^{n}|t_j-s_j|^{2H_0-2} \psi_{t,x}^{(n)}({\bf t},{\bf s}) \ud{\bf t}\ud{\bf s} \\
  & \leq \frac{1}{n!} b_{H_0}^n \left(\sum_{\rho \in S_n} \int_{0<t_{\rho(1)}<\ldots<t_{\rho(n)}<t} \psi_{t,x}^{(n)}({\bf t},{\bf t})^{\frac{1}{2H_0}}\ud{\bf t}\right)^{2H_0}.
\end{align}
By Lemma \ref{lem32-BC}, for any $\rho \in S_n$ fixed, we have:
\begin{align*}
  & \int_{0<t_{\rho(1)}<\ldots<t_{\rho(n)}<t}\psi_{t,x}^{(n)}({\bf t},{\bf t})^{\frac{1}{2H_0}}\ud{\bf t}
  \leq J_0^{1/H_0}(t,x)\int_{0<t_{\rho(1)}<\ldots<t_{\rho(n)}<t} \ud{\bf t} \\
  & \quad \times \left( \int_{\bR^n} \prod_{k=1}^{n}\exp \left\{-\frac{t_{\rho(k+1)}-t_{\rho(k)}}{t_{\rho(k+1)}t_{\rho(k)}}
  \left|\sum_{j=1}^{k}t_{\rho(j)}\xi_j\right|^2\right\}\mu(\ud \xi_1)\ldots \mu(\ud \xi_n)\right)^{\frac{1}{2H_0}}.
\end{align*}
Using the change of variables $t_{k}'=t_{\rho(k)}$ for $k=1,\ldots, n$,
we see that the integral on the right-hand side above does not depend on $\rho$. Hence,
\begin{align*}
       \int_{0<t_{\rho(1)}<\ldots<t_{\rho(n)}<t}\psi_{t,x}^{(n)}({\bf t},{\bf t})^{\frac{1}{2H_0}}\ud{\bf t}
  \leq J_0^{1/H_0}(t,x)  \int_{\{0<t_1<\ldots<t_n<t\}}I_{t}^{(n)}(t_1,\ldots,t_n)^{\frac{1}{2H_0}}\ud{\bf t},
\end{align*}
where
\begin{align*}
    I_{t}^{(n)}(t_1,\ldots,t_n)
  = \int_{\bR^n} \prod_{k=1}^{n}\exp \left\{-\frac{t_{k+1}-t_{k}}{t_{k+1}t_{k}}
    \left|\sum_{j=1}^{k}t_{j}\xi_j\right|^2\right\}\mu(\ud \xi_1)\ldots \mu(\ud \xi_n)
\end{align*}
and $t_{n+1}=t$. Taking the sum over all $\rho \in S_n$ and coming back to \eqref{E:e1}, we obtain:
\begin{equation} \label{I-step1}
       \bE\left(|J_n(t,x)|^2\right)
  \leq J_0^2(t,x) b_{H_0}^n  (n!)^{2H_0-1}
  \left( \int_{\{0<t_1<\ldots<t_n<t\}}I_{t}^{(n)}(t_1,\ldots,t_n)^{\frac{1}{2H_0}}\ud{\bf t}\right)^{2H_0}.
\end{equation}
This inequality is similar to Lemma 3.3 of \cite{balan.chen:18:parabolic}.
\medskip

{\noindent\em Step 2.} We now estimate $I_{t}^{(n)}(t_1,\ldots,t_n)$. We use the change of variables
$z_j=t_j\xi_j$ for $j=1,\ldots,n$, followed by $\eta_k=\sum_{j=1}^k z_j$ for $k=1,\ldots,n$. We
obtain:
\begin{align*}
& I_{t}^{(n)}(t_1,\ldots,t_n)= c_H^n \left(\prod_{i=1}^{n}t_i\right)^{2H-2} \int_{\bR^n}\ud \vec{z} \prod_{k=1}^{n}\exp \left\{-\frac{t_{k+1}-t_{k}}{t_{k+1}t_{k}}
\left|\sum_{j=1}^{k}z_j\right|^2\right\} |z_k|^{1-2H} \\
&=c_H^n\left(\prod_{i=1}^{n}t_i\right)^{2H-2} \int_{\bR^n}\ud \vec{\eta}
\left(\prod_{k=1}^{n}\exp \left\{-\frac{t_{k+1}-t_{k}}{t_{k+1}t_{k}} \left|\eta_k\right|^2\right\}\right)
|\eta_1|^{1-2H}\prod_{k=2}^{n}|\eta_k-\eta_{k-1}|^{1-2H}\\
& \leq
c_H^n \left(\prod_{i=1}^{n}t_i\right)^{2H-2} \int_{\bR^n} \ud \vec{\eta}
\left(\prod_{k=1}^{n}\exp \left\{-\frac{t_{k+1}-t_{k}}{t_{k+1}t_{k}} \left|\eta_k\right|^2\right\} \right)
|\eta_1|^{1-2H}\prod_{k=2}^{n}\left(|\eta_k|^{1-2H}|+|\eta_{k-1}|^{1-2H}\right),
\end{align*}
where $\ud \vec{z}=\ud z_1\cdots\ud z_n$ and similarly $\ud \vec{\eta}=\ud \eta_1\cdots\ud \eta_n$.

By Lemma \ref{L:A} (see also Remark \ref{R:paths} and Figure \ref{F:paths} for more explanations),
\begin{align*}
    |\eta_1|^{1-2H}\prod_{k=2}^{n}(|\eta_k|^{1-2H}|+|\eta_{k-1}|^{1-2H})
  = \sum_{a \in A_n}\prod_{j=1}^{n}|\eta_j|^{(1-2H)a_j}
  = \sum_{\alpha \in D_n}\prod_{j=1}^{n}|\eta_j|^{\alpha_j},
\end{align*}
where $D_n$ is the set of all multi-indices $\alpha=(\alpha_1,\ldots,\alpha_n)$ with
$\alpha_j=(1-2H)a_j$ for all $j=1,\ldots,n$, and $a=(a_1,\ldots,a_n)\in A_n$. Therefore,
\begin{align*}
  I_{t}^{(n)}(t_1,\ldots,t_n) \leq c_H^n \left(\prod_{j=1}^{n}t_j\right)^{2H-2}
  \sum_{\alpha \in D_n}
  \prod_{j=1}^n \left\{\int_{\bR} \exp \Big(-\frac{t_{j+1}-t_{j}}{t_{j+1}t_{j}} |\eta_j|^2
  \Big) |\eta_j|^{\alpha_j} d\eta_j \right\}&.
\end{align*}

Each of the integrals above can be computed explicitly. By Lemma 3.1 of \cite{balan.jolis.ea:15:spde},
\begin{align*}
  \int_{\bR}e^{-t|\xi|^2}|\xi|^{\alpha}\ud \xi=\Gamma\left(\frac{1+\alpha}{2}\right)
  t^{-\frac{1+\alpha}{2}} \quad \text{for any $t>0$ and $\alpha>-1$.}
\end{align*}
Hence,
\begin{align*}
  I_{t}^{(n)}(t_1,\ldots,t_n)
  & \leq C^n \left(\prod_{j=1}^{n}t_j\right)^{2H-2}
    \sum_{\alpha \in D_n}
    \prod_{j=1}^{n} \left( \frac{t_{j+1}-t_j}{t_j t_{j+1}}\right)^{-\frac{1+\alpha_j}{2}}\\
  &=C^n \sum_{\alpha \in D_n} t_1^{\frac{4H-3+\alpha_1}{2}}
    \left(\prod_{j=2}^{n}t_j^{\frac{4H-2+\alpha_{j-1}+\alpha_j}{2}}\right) t^{\frac{\alpha_n+1}{2}}
    \prod_{j=1}^{n} (t_{j+1}-t_j)^{-\frac{\alpha_j+1}{2}}\, .
\end{align*}
\medskip

{\noindent\em Step 3.} Taking power $1/(2H_0)$ and returning to \eqref{I-step1}, we obtain:
\begin{equation} \label{E:norm-f}
  J_n(t,x) \leq J_0^2(t,x) C^n (n!)^{2H_0-1} \left(\sum_{\alpha \in D_n} t^{\frac{\alpha_n+1}{4H_0}}\int_{\{0<t_1<\ldots<t_n<t\}}
    \prod_{i=1}^{n}t_i^{\widetilde{\alpha}_i} (t_{i+1}-t_i)^{\widetilde{\beta}_i}
    \ud{\bf t}\right)^{2H_0}
\end{equation}
where
\begin{align*}
  \widetilde{\alpha}_j =
  \begin{cases}
    \dfrac{4H-3+\alpha_1}{4H_0}              & j=1,\cr
    \dfrac{4H-2+\alpha_{j-1}+\alpha_j}{4H_0} & j=2,\ldots,n,
  \end{cases}
  \quad \text{and} \quad
  \widetilde{\beta}_j =
    -\dfrac{\alpha_j+1}{4H_0}, \quad  j=1,\ldots,n.
\end{align*}

Now we verify that the conditions of Lemma \ref{L:Int} hold for the integrals in \eqref{E:norm-f}.
Clearly, $\widetilde{\alpha}_1>-1$. When $\alpha_j=2(1-2H)$, condition $\widetilde{\beta}_j>-1$
becomes $H+H_0>\sfrac{3}{4}$; see \eqref{E:H-cond}. Now we verify that
\begin{equation} \label{ab-cond}
  \sum_{i=1}^{k} (\widetilde{\alpha}_i+\widetilde{\beta}_i)+k+1+\alpha_{k+1}>0 \quad \text{for all $k=1,\ldots,n-1$},
\end{equation}
using induction on $k$. For $k=1$, using again the condition $H_0+H>\sfrac{3}{4}$, we see that
\begin{align*}
    \widetilde{\alpha}_1+\widetilde{\beta}_1+\widetilde{\alpha}_2+2
  = \frac{8H_0+8H-6+\alpha_1+\alpha_2}{4H_0}
  > 0.
\end{align*}
Suppose that \eqref{ab-cond} holds for $k-1$. We write
\begin{align*}
  \sum_{i=1}^{k} (\widetilde{\alpha}_i+\widetilde{\beta}_i)+k+1+\alpha_{k+1}=
  \left(\sum_{i=1}^{k-1} (\widetilde{\alpha}_i+\widetilde{\beta}_i)+\widetilde{\alpha}_k+k\right)+
  \left(\alpha_{k+1}+\widetilde{\beta}_k+1\right),
\end{align*}
and we notice that $\alpha_{k+1}+\widetilde{\beta}_k+1=(4H_0+4H-3+\alpha_{k+1})/( 4H_0 )>0$.
Therefore, we can apply Lemma \ref{L:Int} to see that
\begin{equation} \label{E:int}
  \begin{gathered}
  \int_{\{0<t_1<\ldots<t_n<t\}}
  \prod_{i=1}^{n}t_i^{\widetilde{\alpha}_i} (t_{i+1}-t_i)^{\widetilde{\beta}_i}
  \ud{\bf t}=\frac{\Gamma\left(\widetilde{\alpha}_1+1\right) \prod_{i=1}^{n}\Gamma(\widetilde{\beta}_i+1)}{\Gamma\big(|\widetilde{\alpha}|+|\widetilde{\beta}|+n+1\big)}
  \gamma_n t^{|\widetilde{\alpha}|+|\widetilde{\beta}|+n}, \\
  \text{with} \quad
  \gamma_n:=\prod_{k=1}^{n-1}\frac{\Gamma\big(\sum_{i=1}^{k}(\widetilde{\alpha}_i+ \widetilde{\beta}_i)+k+1+\widetilde{\alpha}_{k+1}\big) }{\Gamma\big(\sum_{i=1}^{k}(\widetilde{\alpha}_i+\widetilde{\beta}_i)+k+1\big)}.
  \end{gathered}
\end{equation}
\medskip

{\noindent\em Step 4. } In this step, we will show that $\gamma_n\le 1$. Note that
\begin{align*}
  \widetilde{\alpha}_1+\widetilde{\beta}_1 = \frac{H-1}{H_0} \quad \text{and} \quad
  \widetilde{\alpha}_k+\widetilde{\beta}_k = \frac{4H-3+\alpha_{k-1}}{4H_0}, \quad k = 2,\ldots,n.
\end{align*}
Denote $\theta_k:=\sum_{i=1}^k(\widetilde{\alpha}_i+\widetilde{\beta}_i)+k+1$. Hence,
\begin{equation} \label{E:sum-ab}
  \theta_k=
  \begin{cases}
    \displaystyle \frac{H-1}{H_0}+2                                                              & \text{if $k=1$}, \\
    \displaystyle 1-\frac{1}{4H_0}+k\frac{4H_0+4H-3}{4H_0} +\frac{1-2H}{4H_0}\sum_{i=1}^{k-1}a_i & \text{if $k=2,\ldots,n$.}
  \end{cases}
\end{equation}
Note that
\begin{equation} \label{E:theta}
  \theta_{k} - \theta_{k-1} = \frac{4H_0+4H-3}{4H_0}+\frac{1-2H}{4H_0}a_{k-1}, \quad \text{for $k=2,\cdots, n$.}
\end{equation}

We see that $\gamma_n$ is a function of $a_i$:
\begin{align*}
  \gamma_n(a_1,\cdots,a_n) = \prod_{k=1}^{n-1}
  \frac{\Gamma\left(\theta_k+\frac{1-2H}{4H_0}\left(a_k+a_{k+1}-2\right)\right)}{\Gamma\left(\theta_{k}\right)}.
\end{align*}
Recall that any choice of $a_i$ corresponds to a path as shown in Figure \ref{F:a}. We claim that
\begin{align} \label{E:claim}
 \text{when we move the path downwards, the value of $\gamma_n$ decreases.}
\end{align}
As a consequence, the path that achieves the maximum for $\gamma_n$ is the one going through $(i,i)$
for $i=1,\cdots, n$, i.e., the straight diagonal line --- the topmost line. In this case, we have
all $a_i$ are equal to one and hence $a_i+a_{i+1}=2$ for all  $i=1,\cdots, n-1$. Therefore,
\begin{align} \label{gamma-bound}
  \gamma_n(a_1,\cdots,a_n) \le \gamma(1,\cdots,1) = 1.
\end{align}

It remains to prove the claim \eqref{E:claim}. Note that all paths stay between the diagonal and the
line parallel to the diagonal, one unit down. If the path does not touch the diagonal, then no
action is taken (the argument below will show that the value of $\gamma_n$ is minimal for this
path).

Let $(a_1,\cdots,a_n)\in A_n$ be a path which touches the diagonal on at least one point. Say this
point is $(i+1,i+1)$ with $i+1<n$. (The case $i+1=n$ is similar.) We compare the value
$\gamma_n(a_1,\ldots,a_n)$ with the value $\gamma_n(a_1',\ldots,a_n')$ corresponding to another path
$(a_1',\ldots,a_n') \in A_n$ which is obtained by moving the point $(i+1,i+1)$ 1 unit down. There
are 4 possible cases for the shapes of the two paths around the point $(i+1,i+1)$, which are
illustrated in Figure \ref{F:4Cases}. Since $a_k$ gives the number of points that the path
$(a_1,\ldots,a_n)$ has on line $k$, it follows that in all 4 cases, $a_i'=a_i+1$ (since line $i$
received one point), $a_{i+1}'=a_{i+1}-1$ (since line $i+1$ lost one point) and $a_{k}'=a_k$ for all
$k \not \in \{i,i+1\}$ (since the rest of the path remains unchanged). Hence,

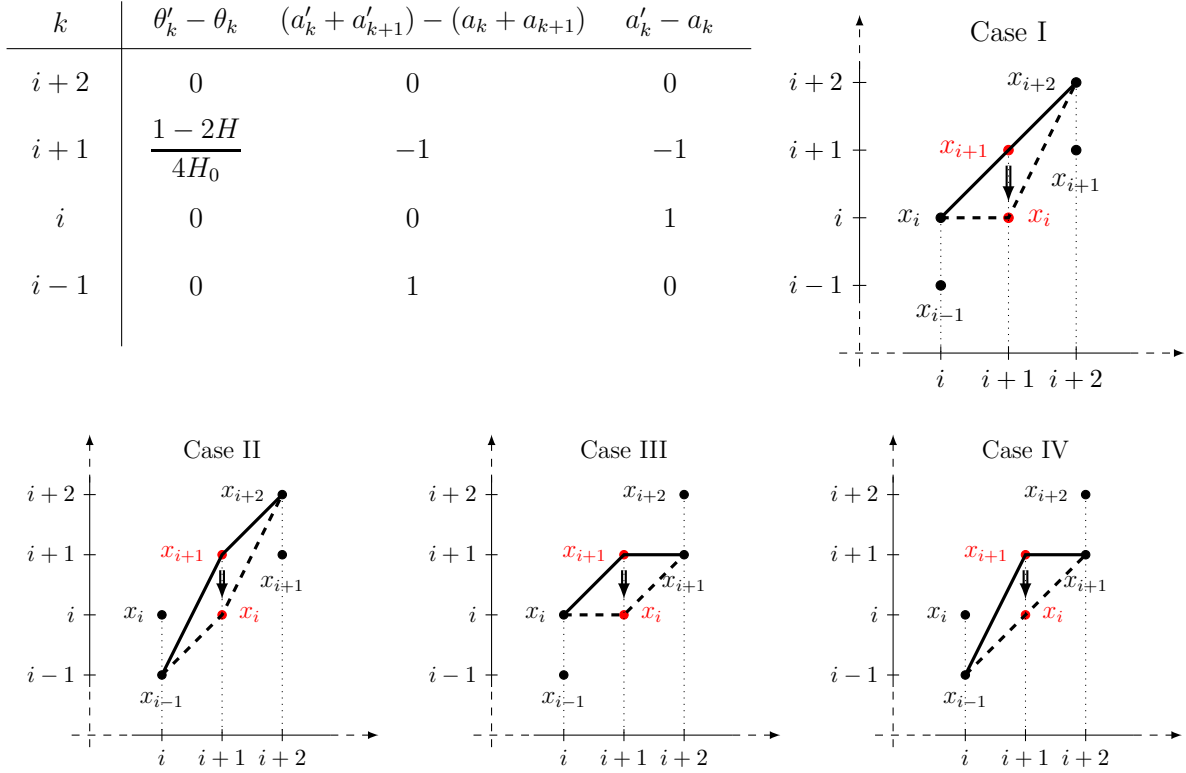
\begin{figure}[htpb]

  \newcommand{\Cases}[1]{
    \filldraw      (0,-1) circle (0.07) node[below,yshift=-4pt] {$x_{i-1}$};
    \filldraw      (0, 0) circle (0.07) node[left, xshift=-4pt] {$x_{i  }$};
    \filldraw[red] (1, 0) circle (0.07) node[right,xshift=+4pt] {$x_{i  }$};
    \filldraw[red] (1, 1) circle (0.07) node[left, xshift=-4pt] {$x_{i+1}$};
    \filldraw      (2, 1) circle (0.07) node[below,yshift=-6pt] {$x_{i+1}$};
    \filldraw      (2, 2) circle (0.07) node[left, xshift=-4pt] {$x_{i+2}$};

    \draw [double, ->, shorten >= 0.5em, shorten <= 0.5em,thick] (1,1) -- (1,0);
    \node at (1,2.75) {Case #1};
    \draw[dashed] (-1.5,-2) --++ (1,0);
    \draw (-0.5,-2) -- (2.8,-2);
    \draw[dashed,->] (2.8,-2) --++ (0.8,0);
    \draw[dashed] (-1.2,-2.2) --++ (0,1);
    \draw (-1.2,-1.2) --++ (0,3.5);
    \draw[dashed,->] (-1.2,2.2) --++ (0,0.8);
    \draw[dotted] (0,0) --++ (0,-2);
    \draw[dotted] (1,1) --++ (0,-3);
    \draw[dotted] (2,2) --++ (0,-4);
    \draw (0,-1.9) --++ (0,-0.2) node[below] {\small $i$};
    \draw (1,-1.9) --++ (0,-0.2) node[below] {\small $i+1$};
    \draw (2,-1.9) --++ (0,-0.2) node[below] {\small $i+2$};
    \draw (-1.1,0) --++ (-0.2,0) node[left]  {\small $i$};
    \draw (-1.1,1) --++ (-0.2,0) node[left]  {\small $i+1$};
    \draw (-1.1,2) --++ (-0.2,0) node[left]  {\small $i+2$};
    \draw (-1.1,-1)--++ (-0.2,0) node[left]  {\small $i-1$};
  }
  \begin{center}
    \begin{tikzpicture}[scale=0.9, transform shape]
      \tikzset{>=latex}
      \Cases{I}

      \node at (-4,2.9) {$a_k'-a_k$};
      \node at (-7.5,2.9) {$(a_k'+a_{k+1}')-(a_k+a_{k+1})$};
      \draw (-13.8,2.5) --++ (11,0);
      \draw (-12.1,3.2) --++(0,-5.1);
      \node at (-11,2.9) {$\theta_k'-\theta_{k}$};
      \node at (-13,2.9) {$k$};

      \node at (-4,+2) {$0$};
      \node at (-4,+1) {$-1$};
      \node at (-4, 0) {$1$};
      \node at (-4,-1) {$0$};

      \node at (-7.8,+2) {$0$};
      \node at (-7.8,+1) {$-1$};
      \node at (-7.8, 0) {$0$};
      \node at (-7.8,-1) {$1$};

      \node at (-11,+2) {$0$};
      \node at (-11,+1) {$\dfrac{1-2H}{4H_0}$};
      \node at (-11, 0) {$0$};
      \node at (-11,-1) {$0$};

      \node at (-13,+2) {$i+2$};
      \node at (-13,+1) {$i+1$};
      \node at (-13, 0) {$i$};
      \node at (-13,-1) {$i-1$};

      \draw[very thick] (0,0) -- (1,1) -- (2,2);
      \draw[dashed,very thick] (0,0) -- (1,0) -- (2,2);
    \end{tikzpicture}
    \bigskip

    \begin{tikzpicture}[scale=0.8, transform shape]
      \tikzset{>=latex}
      \Cases{II}
      \draw[very thick] (0,-1) -- (1,1) -- (2,2);
      \draw[dashed,very thick] (0,-1) -- (1,0) -- (2,2);
    \end{tikzpicture}
    \quad
    \begin{tikzpicture}[scale=0.8, transform shape]
      \tikzset{>=latex}
      \Cases{III}
      \draw[very thick] (0,0) -- (1,1) -- (2,1);
      \draw[dashed,very thick] (0,0) -- (1,0) -- (2,1);
    \end{tikzpicture}
    \quad
    \begin{tikzpicture}[scale=0.8, transform shape]
      \tikzset{>=latex}
      \Cases{IV}
      \draw[very thick] (0,-1) -- (1,1) -- (2,1);
      \draw[dashed,very thick] (0,-1) -- (1,0) -- (2,1);
    \end{tikzpicture}
  \end{center}
  \caption{Four cases for the positions of the paths $(a_1,\ldots,a_n)$ and $(a_1',\ldots,a_n')$
  around point $(i+1,i+1)$. The values of $a_k'-a_k$, $(a_k'+a_{k+1}')-(a_k+a_{k+1})$, and
$\theta_k'-\theta_k$ are the same for all four cases.}

  \label{F:4Cases}
\end{figure}

\begin{align*}
  (a_k'+a_{k+1}')-(a_k+a_{k+1})=
  \begin{cases}
    1  & \text{if $k=i-1$} \\
    -1 & \text{if $k=i+1$} \\
    0  & \text{otherwise}
  \end{cases}
  \quad \text{and} \quad
 \theta_k'-\theta_k=
 \begin{cases}
   \frac{1-2H}{4H_0} & \text{if $k=i+1$} \\
    0                & \text{otherwise}
  \end{cases}.
\end{align*}

By direct calculation, we see that
\begin{align*}
  \theta_{i-1}'+\frac{1-2H}{4H_0}(a_{i-1}'+a_{i}'-2) & = \theta_{i-1}+\frac{1-2H}{4H_0}(a_{i-1}+a_{i}-2)+\frac{1-2H}{4H_0} \quad \text{and} \\
  \theta_k'+\frac{1-2H}{4H_0}(a_{k}'+a_{k+1}'-2)     & = \theta_k+\frac{1-2H}{4H_0}(a_{k}+a_{k+1}-2)
  \quad \text{for all $k \not= i-1$}.
\end{align*}
Therefore,
\begin{align*}
  \frac{\gamma\left(a_1',\cdots,a_n'\right)}{\gamma\left(a_1,\cdots,a_n\right)}
 & = \prod_{k=1}^{n-1}\dfrac{\Gamma(\theta_k'+\frac{1-2H}{4H_0}(a_k'+a_{k+1}'-2))}{\Gamma(\theta_k+ \frac{1-2H}{4H_0}(a_k+a_{k+1}-2))}
     \times \prod_{k=1}^{n-1}\dfrac{\Gamma(\theta_k)}{\Gamma(\theta_k')}\\
 & = \frac{\Gamma\left(\theta_{i-1}+\frac{1-2H}{4H_0}(a_{i-1}+a_{i}-2)+\frac{1-2H}{4H_0}\right)}{\Gamma\left(\theta_{i-1}+ \frac{1-2H}{4H_0}(a_{i-1}+a_{i}-2)\right)}
     \times \frac{\Gamma(\theta_{i+1})}{\Gamma\left(\theta_{i+1}+\frac{1-2H}{4H_0}\right)}\:.
\end{align*}

By applying Lemma \ref{L:Gamma}, we see that the above ratio is always less than or equal to one.
For this, we need to check that
\begin{align*}
  z_1:=\theta_{i-1}+\frac{1-2H}{4H_0}(a_{i-1}+a_i-2)\leq z_2:=\theta_{i+1}.
\end{align*}
This is clear, since by \eqref{E:theta}, $\theta_{i+1}=\theta_{i-1}+2
\frac{4H_0+4H-3}{4H_0}+\frac{1-2H}{4H_0}(a_{i-1}+a_i-2)$. Here we use again condition
\eqref{E:H-cond}. This proves the claim in \eqref{E:claim}. \bigskip

{\noindent\em Step 5.}  We claim that for $n$ large enough,
\begin{equation} \label{E:LB}
  \Gamma\left(|\widetilde{\alpha}|+|\widetilde{\beta}|+n+1\right) \geq C^n (n!)^{\frac{2H_0+H-1}{2H_0}}.
\end{equation}
Indeed, by \eqref{E:sum-ab},
\begin{align*}
    |\widetilde{\alpha}|+|\widetilde{\beta}|+n+1
  = n\frac{2H_0+H-1}{2H_0}-\frac{1+\alpha_n}{4H_0}+1 \geq n\frac{2H_0+H-1}{2H_0}-\frac{1-H}{2H_0}+1.
\end{align*}
We use the fact that for any $a>0,b \in \bR$, there exists $N_{a,b} \in \bN$ depending on $a$ and
$b$ such that $\Gamma(an+1+b) \geq C_{a,b}^n (n!)^a$ for all $n \geq N_{a,b}$. Since $\Gamma$ is
increasing on $(2,\infty)$, we see that \eqref{E:LB} holds true. Therefore, thanks to \eqref{E:LB},
we see that for $n$ large enough,
\begin{align*}
  \int_{\{0<t_1<\ldots<t_n<t\}}
  \prod_{i=1}^{n}t_i^{\widetilde{\alpha}_i} (t_{i+1}-t_i)^{\widetilde{\beta}_i} \ud{\bf t}
  \leq
  C^n (n!)^{-\frac{2H_0+H-1}{2H_0}}
  t^{n \frac{2H_0+H-1}{2H_0}-\frac{1+\alpha_n}{4H_0}}.
\end{align*}
Returning to \eqref{E:norm-f}, it follows that $\bE\left(|J_n(t,x)|^2\right) \leq J_0^{2}(t,x)C^n
(n!)^{-H} t^{n (2H_0+H-1)}$. Finally, by hypercontractivity, the $\|\cdot\|_p$-norm on $L^p(\Omega)$
is equivalent to the $\|\cdot\|_2$-norm (see e.g. page 62 of \cite{nualart:06:malliavin}), and hence
\begin{align*}
    \|u(t,x)\|_p
  & \leq \sum_{n\geq 0}(p-1)^{n/2}\|J_n(t,x)\|_2 \leq J_0(t,x) \sum_{n\geq 0}(p-1)^{n/2}C^{n/2} \frac{1}{(n!)^{H/2}}t^{n \frac{2H_0+H-1}{2} } \\
  & \leq C \exp\left(C p^{1/H}t^{\frac{2H_0+H-1}{H}} \right),
\end{align*}
where for the last line we used the fact that $\sum_{n\geq 0}\frac{x^n}{(n!)^a} \leq C \exp(c
x^{1/a})$ for any $x>0$ and $a>0$. We conclude the proof of Theorem \ref{T:main} by taking power
$p$.
\end{proof}

\appendix
\section{Some auxiliary results}
\begin{lemma} \label{L:Int}
  Suppose that $\alpha_1>-1$, $\beta_i>-1$ for any $i=1,\ldots,n$, and
  \begin{equation} \label{cond-C}
    \sum_{i=1}^{k} (\alpha_i+\beta_i)+k+1+\alpha_{k+1}>0 \quad \mbox{for all} \ k=1,\ldots,n-1.
  \end{equation}
  Then by setting $t_{n+1}=t$, $|\alpha|=\sum_{i=1}^{n}\alpha_i$ and
  $|\beta|=\sum_{i=1}^{n}\beta_i$, we have that
  \begin{align} \label{lemA-id}
    I_n & (t,\alpha_1,\ldots,\alpha_n,\beta_1,\ldots,\beta_n):=\int_{\{0<t_1<\ldots<t_n<t\}} \prod_{i=1}^{n}t_i^{\alpha_i}(t_{i+1}-t_i)^{\beta_i}\ud{\bf t} \nonumber \\
        & =\frac{\Gamma(\alpha_1+1) \prod_{i=1}^{n}\Gamma(\beta_i+1)}{\Gamma\big(|\alpha|+|\beta|+n+1\big)}
           \prod_{k=1}^{n-1}\frac{\Gamma\big(\sum_{i=1}^{k}(\alpha_i+\beta_i)+k+1+\alpha_{k+1}\big) }{\Gamma\big(\sum_{i=1}^{k}(\alpha_i+\beta_i)+k+1\big)} t^{|\alpha|+|\beta|+n}.
  \end{align}
\end{lemma}
\begin{proof}
  The lemma is proved by induction. For $n=1$, we have:
  \begin{align*}
    I_1(t,\alpha,\beta)=\int_0^t t_1^{\alpha_1}(t-t_1)^{\beta_1}\ud t_1=\frac{\Gamma(\alpha_1+1)
    \Gamma(\beta_1+1)}{\Gamma(\alpha_1+\beta_1+2)}t^{\alpha_1+\beta_1+1}.
  \end{align*}
  For the induction step, we use the fact that
  \begin{align*}
    I_n(t,\alpha_1,\ldots,\alpha_n,\beta_1,\ldots,\beta_n)=\int_0^t t_n^{\alpha_n}(t-t_n)^{\beta_n} I_{n-1}(t_n,\alpha_1,\ldots,\alpha_{n-1},\beta_1,\ldots,\beta_{n-1})\ud t_n.
  \end{align*}
  This proves the lemma.
\end{proof}

\begin{lemma} \label{L:Gamma}
  For any $a>0$, the function $z \mapsto \Gamma(z+a)/\Gamma(z)$ is non-decreasing on $(0,\infty)$.
\end{lemma}
\begin{proof}
  Let $f(z)=\Gamma(z+a)/\Gamma(z)$. Note that $f'(z)=f(z) \left(\psi(z+a)-\psi(z)\right)$, where
  $\psi(z):=\Gamma'(z) / \Gamma(z)$ is the {\it psi function}; see 5.2.2 on p. 136 of
  \cite{olver.ea:10:NIST}. By the following expression
  \begin{align*}
    \psi(z)=-\gamma+\sum_{k\geq 1}\left( \frac{1}{k+1}-\frac{1}{k+z}\right) \quad \mbox{for any} \quad z>0,
  \end{align*}
  with $\gamma=\lim_{n\to \infty}(\sum_{k=1}^n k^{-1}-\ln n) \approx 0.5772$ being the Euler's
  constant (see, e.g., 5.7.6 on p. 139 {\it ibid.}), we see that $\psi$ is a nondecreasing function
  on $(0,\infty)$. Hence, $f'(z)\ge 0$ for all $z>0$, which implies the desired result.
\end{proof}


\begin{thebibliography}{99}
  \bibitem{amir.corwin.ea:11:probability}
  Amir, G., Corwin, I., and Quastel J. (2011).
  \newblock Probability distribution of the free energy of the continuum directed random polymer in {$1+1$} dimensions.
  \newblock {\em Comm. Pure Appl. Math.}, {\bf 64}, no. 4, 466--537.
  \bibitem{balan.chen:18:parabolic}
  Balan, R.M. and Chen, L. (2018).
  \newblock Parabolic Anderson Model with space-time homogeneous Gaussian noise and rough initial condition.
  \newblock {\em J. Theoret. Probab.} {\bf 31}, 2216--2265.
  \bibitem{balan.jolis.ea:15:spde}
  Balan, R.M., Jolis, M. and Quer-Sardanyons, L. (2015).
  \newblock SPDEs with affine multiplicative fractional noise in space with index $H \in (1/4,1/2)$.
  \newblock {\em Electr. J. Probab.} {\bf 20}, no. 54, 1--36.
  \bibitem{chen.dalang:15:moments}
  Chen, L. and Dalang, R.C. (2015).
  \newblock Moments and growth indices for the nonlinear stochastic heat equation with rough initial conditions.
  \newblock {\em Ann. Probab.}, {\bf 43}, No. 6, 3006--3051.
  \bibitem{chen.huang:19:comparison}
  Chen, L. and Huang, J. (2019).
  \newblock Comparison principle for stochastic heat equation on $\R^d$.
  \newblock {\em Ann.\ Probab.}  {\bf 47}, no. 2, 989--1035.
  \bibitem{chen.kim:19:nonlinear} 
  Chen, L. and Kim, K. (2019).
  \newblock Nonlinear stochastic heat equation driven by spatially colored noise: moments and intermittency.
  \newblock {\em  Acta Math. Sci. Ser. B}, \textbf{39}(3),  645--668.
  \bibitem{chen:19:parabolic} 
  Chen, X. (2019).
  \newblock Parabolic Anderson model with rough or critical Gaussian noise.
  \newblock {\em Ann. Inst. Henri Poincar\'e: Prob. Stat.} {\bf 55}, 941--976.
  \bibitem{hu.le:19:joint}
  Hu, Y. and L\^{e}, K. (2019).
  \newblock Joint H\"older continuity of parabolic Anderson model.
  \newblock {\em Acta Math. Sci.} {\bf 39}, 764--780.
  \bibitem{hu.huang.ea:18:parabolic}
  Hu, Y., Huang, J., L\^e, K., Nualart, D. and Tindel, S. (2018).
  \newblock Parabolic Anderson model with rough dependence in space.
  \newblock In: ``{\it Abel Symposia 2016: Computations and Combinatorics in Dynamics, Stochastics and Control}'', {\bf 16}, 477--498. Springer, Cham.
  \bibitem{memin.mishura.ea:01:inequalities}
  M\'{e}min, J., Mishura, Y. and Valkeila, E. (2001).
  \newblock Inequalities for the moments of Wiener integrals with respect to a fractional Brownian motion.
  \newblock {\it Statist. Probab. Lett.} {\bf 51}, No. 2, 197--206.
  \bibitem{nualart:06:malliavin}
  Nualart D. (2006).
  \newblock {\em The Malliavin Calculus and Related Topics}. Second edition.
  \newblock Springer-Verlag, Berlin.
  \bibitem{olver.ea:10:NIST}
  Olver, F.~W.~J., Lozier, D.~W., Boisvert, R.~F. and Clark, C.~W. (2010).
  \newblock {\em N{IST} handbook of mathematical functions}.
  \newblock U.S. Department of Commerce National Institute of Standards and Technology, Washington, DC. Cambridge Univ. Press, Cambridge.
  \bibitem{rudin:91:functional}
  Rudin, W. (1991).
  \newblock {\it Functional analysis. } Second edition.
  \newblock International Series in Pure and Applied Mathematics. McGraw-Hill, Inc., New York.
  \bibitem{song.song.ea:20:fractional}
  Song, J., Song, X. and Xu, F. (2020).
  \newblock Fractional stochastic wave equation driven by a Gaussian noise rough in space.
  \newblock {\em Bernoulli} {\bf 26}, 2699--2726.
\end{thebibliography}
\end{document}